\documentclass[12pt]{article}
\usepackage{enumerate}
\usepackage{amsmath, amssymb, amsthm, amsfonts}
\usepackage{geometry, graphicx}
\newtheorem{theorem}{Theorem}[section]
\newtheorem{proposition}[theorem]{Proposition}

\newtheorem{lemma}[theorem]{Lemma}

\theoremstyle{definition}
\newtheorem{remark}[theorem]{Remark}

\newtheorem{question}[theorem]{Question}

\def\IQ{\mathbb{Q}}

\def\IN{\mathbb{N}}
\def\IR{\mathbb{R}}

\def\P{\mathcal{P}}
\def\Ph{\hat{\P}}
\def\Q{\mathcal{Q}}
\def\RG{\mathcal{RG}}

\def\s2{\{0, 1\}}

\newcommand{\mathdef}{\stackrel{\textrm{\scriptsize def}}{=}}

\newcommand{\cl}{\mbox{cl}}
\newcommand{\conv}{\mbox{conv}}

\newcommand{\ct}{\widetilde{c}_4}
\newcommand{\Pt}{\widetilde{P}}

\begin{document}

\title{On the density of triangles and squares\\
in regular finite and unimodular random graphs}


\author{Viktor Harangi}

\date{}

\maketitle
\begin{abstract}
We explicitly describe the possible pairs of triangle and square densities
for $r$-regular finite simple graphs.
We also prove that every $r$-regular unimodular random graph
can be approximated by $r$-regular finite graphs with respect to these densities.
As a corollary one gets an explicit description of the possible pairs
of the third and fourth moments of the spectral measure of
$r$-regular unimodular random graphs.
\bigskip



\end{abstract}

\noindent
\textsc{A.~R\'enyi Institute of Mathematics,
\\Hungarian Academy of Sciences, \\
P.O.B.~127, H-1364 Budapest, Hungary}\\
\textit{Email address}: \verb+harangi@gmail.com+

\bigskip
\noindent
\emph{Acknowledgment:} The author was supported by
Hungarian Scientific Foundation grant no.~72655.


\section{Introduction}
For a finite simple graph $G$ with vertex set $V(G)$
let $c_k(G)$ denote the total number of $k$-cycles (i.e., cycles of length $k$) in $G$.
We define the \textit{density of $k$-cycles} in $G$ as
$$ d_k = d_k(G) \mathdef \frac{c_k(G)}{ \left| V(G)) \right| } .$$
In this paper we explicitly describe the possible pairs $(d_3, d_4)$
for $r$-regular (not necessarily connected) 
finite graphs for any fixed integer $r \geq 3$.
In other words, we describe the relation between the normalized number of
three- and four-cycles in regular graphs.
Problems of similar nature have been studied in the literature.
For instance, the following question of Erd\H os has been settled only recently:
given an arbitrary graph $G$ on $n$ vertices,
what is the maximal number of $5$-cycles if there is no $3$-cycle in $G$?
For details, see \cite{gyori, razborov_et_al}.

Our method actually works in a more general setting:
not only for finite graphs but also for unimodular random graphs.
A random graph is a probability distribution
on the space of locally finite, connected rooted graphs.
A random graph is unimodular if it satisfies a certain invariance property
that will be explained in Section \ref{sec:unimodular}.
The study of unimodular random graphs was initiated
by Aldous and Lyons in \cite{aldous_lyons}.
This notion drew a lot of attention partially because of
its connection to Benjamini-Schramm convergence.
The limit of a convergent sequence of finite graphs is
a random graph that is necessarily unimodular.
One of the most intriguing open questions in the area is
whether the condition of unimodularity is also sufficient.
\begin{question}[Question 10.1 of \cite{aldous_lyons}] \label{q:big_open}
Can every unimodular random graph
be obtained as the Benjamini-Schramm limit of finite graphs?
\end{question}
%
%
Mikl\'os Ab\'ert suggested the following approach.
One can naturally define the densities $d_k$ for a random graph
($d_k$ is the expected number of $k$-cycles containing the root divided by $k$).
Thus for some integers $k_1, \ldots, k_m$ the tuple
\begin{equation*} 
\left( d_{k_1}, d_{k_2}, \ldots, d_{k_m} \right)
\end{equation*}
can be associated both to finite graphs and to random graphs.
It is easy to see that if a sequence $G_n$ of finite graphs
converges to a random graph $G$, then $d_k(G_n)$ tends to $d_k(G)$ for any fixed k.
So it is natural to ask whether any $r$-regular unimodular random graph $G$
can be approximated by $r$-regular finite graphs
in the sense that the tuples of the finite graphs
converge to the tuple corresponding to $G$.
If we found a unimodular random graph that cannot be
approximated in the above sense for some $k_1, \ldots, k_m$, then
this random graph would be impossible to be obtained
as the Benjamini-Schramm limit of finite graphs.
We settle the question for the first non-trivial case $(d_3, d_4)$.
\begin{theorem}
The set
$$ \P^r_{3,4} \mathdef \left\{ \left( d_3(G), d_4(G) \right) :
\mbox{$G$ is an $r$-regular finite simple graph} \right\} \subset \IR^2 $$
is dense in the set
$$ \Ph^r_{3,4} \mathdef \left\{ \left( d_3(G), d_4(G) \right) :
\mbox{$G$ is an $r$-regular unimodular random graph} \right\} \subset \IR^2 .$$
In fact, we give an explicit description of
the sets $\P^r_{3,4}$ and $\Ph^r_{3,4}$ for any given $r$.
\end{theorem}
The above results can be interpreted in the language of spectral measures as well.
For a finite graph $G$ consider the transition matrix $M_G$
of the simple random walk on $G$.
The set of eigenvalues of $M_G$ is called the spectrum of $G$.
One can get a probability measure on the spectrum
by putting mass $1/|V(G)|$ at each eigenvalue (counting multiplicities).
We call this measure the eigenvalue distribution of $G$
and denote it by $\mu_G$.

A related probability measure (the so-called spectral measure) can be
associated to locally finite, connected rooted graphs;
for details, see Section \ref{sec:unimodular}.
From our point of view one of the most important features
of these measures is the following: if $G_n$ is a graph sequence
converging to a random graph $G$,
then the measures $\mu_{G_n}$ weakly converge
to the expected spectral measure of $G$, which we will also denote by $\mu_G$.

\begin{question}[Ab\'ert] \label{q:abert1}
Can the expected spectral measure of every unimodular random graph
be obtained as the weak limit of the eigenvalue distribution of finite graphs?
\end{question}

Of course, a negative answer to this question would imply
a negative answer to Question \ref{q:big_open}.
When studying these measures, it is natural to look at their moments:
the $k$-th moment of $\mu_G$
is defined as $\int x^k \, \mathrm{d} \mu_G(x)$.
For $k=3,4$ these $k$-th moments can be computed from the densities $d_k$.
Thus as a corollary of our result we get a description of the possible pairs of
third and fourth moments of $\mu_G$ for $r$-regular finite graphs
and also for $r$-regular unimodular random graphs.
Note that in a recent paper Ab\'ert, Glasner and Vir\'ag
studied the possible shapes of the eigenvalue distribution of
$r$-regular finite graphs \cite{abert_glasner_virag}.

\subsection*{Acknowledgments}
The author is grateful to Mikl\'os Ab\'ert for suggesting the problem,
and to P\'eter Csikv\'ari for useful comments.


\section{Finite graphs} \label{sec:finite}

In this section we study the density of triangles and squares
in regular finite graphs.

\subsection{Notations and preliminaries}

Let $G$ be an $r$-regular finite simple graph for a fixed integer $r \geq 3$;
$V(G)$ denotes the set of its vertices, $E(G)$ is the set of its edges.
First we explain how the third and fourth moments
of the eigenvalue distribution of $G$
can be expressed in terms of the density of three- and four-cycles in $G$.
Recall that $M_G$ denotes the transition matrix
of the simple random walk on $G$;
let the eigenvalues of $M_G$ be $\lambda_1, \ldots, \lambda_{|V(G)|}$.
The eigenvalue distribution of $G$ is the following probability measure:
$$ \mu_G \mathdef \frac{1}{|V(G)|}\sum_{i=i}^{|V(G)|} \delta_{\lambda_i} ,$$
where $\delta_x$ is the Dirac measure centered on $x$.
The $k$-th moment of this measure is
$$ \int x^k \, \mathrm{d} \mu_G(x)=
\frac{1}{|V(G)|} \sum_{i=1}^{|V(G)|} \lambda_i^k .$$
The eigenvalues of the $k$-th power of the transition matrix are
$\lambda_i^k$; $i=1,\ldots, |V(G)|$. So their sum is equal to the trace of $M_G^k$,
which, in turn, can be expressed in terms of return probabilities.
Consider the simple random walk on $G$ starting from a vertex $x$.
Let $p_k(G,x)$ denote the probability of return in $k$ steps.
It is easy to see that these return probabilities $p_k(G,x)$
are the elements of the main diagonal of the matrix $M_G^k$.
So the $k$-th moment in question is simply
the average of these return probabilities:
$$ \int x^k \, \mathrm{d} \mu_G(x) =
\frac{1}{|V(G)|} \sum_{x \in V(G)} p_k(G,x) .$$
For $k=3,4$, $p_k(G,x)$ is determined by
the number of $k$-cycles containing $x$,
let us denote this number by $c_k(G,x)$. Clearly,
\begin{equation*} 
p_3(G,x) = \frac{2 c_3(G,x)}{r^3} \ ; \
p_4(G,x) = \frac{2 c_4(G,x) + 2r^2 - r }{r^4} .
\end{equation*}
Let the total number of $k$-cycles in $G$ be $c_k(G)$; then
$$ c_k(G) = \frac{1}{k} \sum_{x \in V(G)} c_k(G,x) .$$
We define the density of the $k$-cycles as
$$ d_k(G) = \frac{c_k(G)}{|V(G)|} .$$
It follows that
\begin{equation*} 
\int x^3 \, \mathrm{d} \mu_G(x) = \frac{6}{r^3}d_3(G) \ ; \
\int x^4 \, \mathrm{d} \mu_G(x) =
\frac{8}{r^4}d_4(G) + \frac{2r-1}{r^3} .
\end{equation*}
Consequently, determining the possible pairs of the third and fourth moments of $\mu_G$
is equivalent to describing the possible pairs $\left( d_3(G), d_4(G) \right)$.

We introduce the following notations.
For a graph $G$ and a vertex $x \in V(G)$ we set
\begin{align*}
P(G,x) = P_{3,4}(G,x) \mathdef
\left( c_3(G,x)/3, c_4(G,x)/4 \right) \in \IR^2, \\
P(G) = P_{3,4}(G) \mathdef \left( d_3(G), d_4(G) \right) \in \IR^2 .
\end{align*}
By definition, $P(G)$ is the center of the points $P(G,x); x \in V(G)$.
Which points of $\IR^2$ can we get as $P(G)$ for some $r$-regular graph $G$?
How does the set of such points look like? Let
$$ \P^r = \P^r_{3,4} \mathdef \left\{ P_{3,4}(G) :
\mbox{$G$ is an $r$-regular finite simple graph} \right\} \subset \IR^2 .$$
The following simple observation shows that the closure of this set is convex.
\begin{proposition} \label{prop:convex}
For any two points $P_1,P_2 \in \P^r$ and any rational number $0<q<1$,
the convex combination $q P_1 + (1-q) P_2$ also lies in $\P^r$.
Consequently, $\cl(\P^r)$ must be convex.
\end{proposition}
\begin{proof}
Take an $r$-regular $G_i$ with $P(G_i) = P_i$, $i=1,2$.
Let $a<b$ be positive integers with $q = a/b$.
Let $G$ be the disjoint union of $a|V(G_2)|$ copies of $G_1$
and $(b-a)|V(G_1)|$ copies of $G_2$.
Clearly, $G$ is $r$-regular and $P(G) = q P_1 + (1-q) P_2$.
\end{proof}
How does this compact convex set $\cl(\P^r)$ look like for a given $r$?
Using a result of Bollob\'as we managed to fully describe this convex set.
We will prove that it is a convex polygon with $\lceil r/2 \rceil +2$ vertices.
In fact, $\P^r$ consists of those points in this polygon
which have rational coordinates.

We need to introduce a few more notations.
In an $r$-regular graph $G$ we denote
the set of neighbors of a vertex $x$ by $N(x)$,
and the induced subgraph on $N(x)$ by $G^x$. Clearly,
$$ \left| V(G^x) \right| = \left| N(x) \right| = r \ ; \
\left| E(G^x) \right| = c_3(G,x) .$$

For the sake of simplicity we will use
the term \textit{cherry} for paths of length $2$.
We will call the first and last vertex of the path \textit{leaves},
while the middle vertex will be referred to as the \textit{node} of the cherry.
It is easy to see that $c_4(G,x)$ equals the number of cherries in $G$
with both leaves lying in $N(x)$ and the node being different from $x$.
We distinguish two types of such cherries depending on whether
the node also lies in $N(x)$ (type 1) or not (type 2).

Suppose that for a given vertex $x$ the graph $G^x$
has degrees $d_1, \ldots, d_r \geq 0$.
Then the number of type 1 cherries with respect to $x$
equals the number of cherries in $G^x$:
$$ \sum_{i=1}^{r} \binom{d_i}{2} .$$
The type 2 cherries are those having both leaves in $N(x)$ and
their node outside $\{x\} \cup N(x)$.
Let us consider those edges of $G$ which have one endpoint in $N(x)$
and one outside $\{x\} \cup N(x)$.
If each of these edges has a different endpoint outside $\{x\} \cup N(x)$,
then there are no type 2 cherries with respect to $x$.
Soon we will see examples of regular graphs that do not have
any type 2 cherries with respect to any of its vertices.

\subsection{Constructing the extreme graphs} \label{sec:constr}

In this section we will construct the \textit{extreme graphs} of our problem,
that is, the graphs corresponding to the vertices of the polygon $\cl(\P^r)$.
Consider an arbitrary partition of $r$:
$$ r = r_1 + \cdots + r_l $$
for some positive integers $l, r_1, \ldots, r_l$.
Due to Lemma \ref{lemma:reg_hyp} of the Appendix there exists a hypergraph $H$
such that each of its vertices is contained by
$l$ hyperedges with sizes $r_1+1, \ldots, r_l+1$,
and the girth of $H$ is at least $5$
(that is, any Berge cycle has length at least $5$).
Now the graph $G=G_{r_1,\ldots,r_l}$ is obtained from $H$ by the following way:
$V(G) = V(H)$ and two vertices are connected with an edge in $G$
if there is a hyperedge in $H$ which contains both of them.

It can be seen easily that $G$ is an $r$-regular graph with the property
that for each $x \in V(G)$ the graph $G^x$ is isomorphic to
the disjoint union of the complete graphs $K_{r_1}, \ldots, K_{r_l}$.
Also, $G$ does not contain any type 2 cherries,
which follows easily from the fact that $H$ has girth at least $5$.

How can one compute the values of $d_3(G)$ and $d_4(G)$ for such a graph $G$?
It is quite easy because $G$ was constructed in such a way that
$c_3(G,x)$ and $c_4(G,x)$ do not depend on $x \in V(G)$:
\begin{equation*} 
c_3(G,x) = \sum_{i=1}^l \binom{r_i}{2} \ ; \
c_4(G,x) = \sum_{i=1}^l r_i \binom{r_i-1}{2} .
\end{equation*}
It follows that $d_3(G) = c_3(G,x)/3$ and $d_4(G) = c_4(G,x)/4$
for an arbitrary vertex $x \in V(G)$.

We will use these graphs only in the special case
when the $r_i$'s are (almost) equal. 
More precisely, for a positive integer $l \leq r$
we take the partition of $r$ into $l$ parts
with as equal parts as possible; namely, we set
$$ r_1 = \left[ \frac{r}{l} \right] \ ; \
 r_2 = \left[ \frac{r+1}{l} \right] \ ; \ \ldots \ ; \
  r_l = \left[ \frac{r+l-1}{l} \right] .$$
We use these values in the above construction; let $C^r_l$ denote the obtained graph.
We know that all the induced subgraphs $(C^r_l)^x$ are isomorphic
to the $r$-vertex graph which is the disjoint union of $l$ complete graphs
with sizes as equal as possible; let this graph be $D^r_l$.
The complement of $D^r_l$ is the Tur\'an graph $T^r_l$:
the unique graph that has the maximum possible number of edges
of any $r$-vertex graph not containing a complete graph $K_{l+1}$.

Also note that the graph $C^r_1$ is isomorphic to
the complete graph $K_{r+1}$ (or the disjoint union of $K_{r+1}$'s),
while $C^r_r$ has girth at least $5$, that is,
it does not contain any three- or four-cycles.

For a fixed $r$ let us consider these graphs $C^r_1, \ldots, C^r_r$
along with the complete bipartite graph $K_{r,r}$.
(For notational convenience we set $C^r_0 \mathdef K_{r,r}$.)
Each of these graphs is $r$-regular, so
the corresponding points $P^r_l \mathdef P(C^r_l)$ are in $\P^r$.
It follows that the convex hull of these points is contained by $\cl(\P^r)$:
$$ \cl(\P^r) \supset \Q^r \mathdef \conv \left\{ P^r_0, P^r_1, \ldots, P^r_r \right\}
= \conv \left\{ P(K_{r,r}), P(C^r_1), \ldots, P(C^r_r) \right\} .$$
In fact, $\cl(\P^r) = \Q^r$.
Before we prove that, let us examine how this polygon $\Q^r$ looks like.
Actually, one does not need to consider all the points $P^r_l$.
For $\lceil  r/2 \rceil \leq l \leq r$, the graph $C^r_l$ contains
no four-cycles, so $d_4(C^r_l) = 0$.
It means that the points $P^r_l = P(C^r_l)$, $\lceil  r/2 \rceil \leq l \leq r$
all lie on the horizontal segment connecting
$P^r_{\lceil r/2 \rceil}$ and $P^r_r$.
Thus it is enough to consider the convex hull of the points
$P^r_0, P^r_1, \ldots, P^r_{\lceil r/2 \rceil}, P^r_r$.
However, it is not hard to see that
one cannot omit any more points from this system:
$\Q^r$ is a convex polygon with $\lceil r/2 \rceil +2$ vertices.
Figure \ref{fig:q12} shows the polygon $\Q^{12}$.
\begin{figure}[h]
\centering
\includegraphics[width=10.5cm]{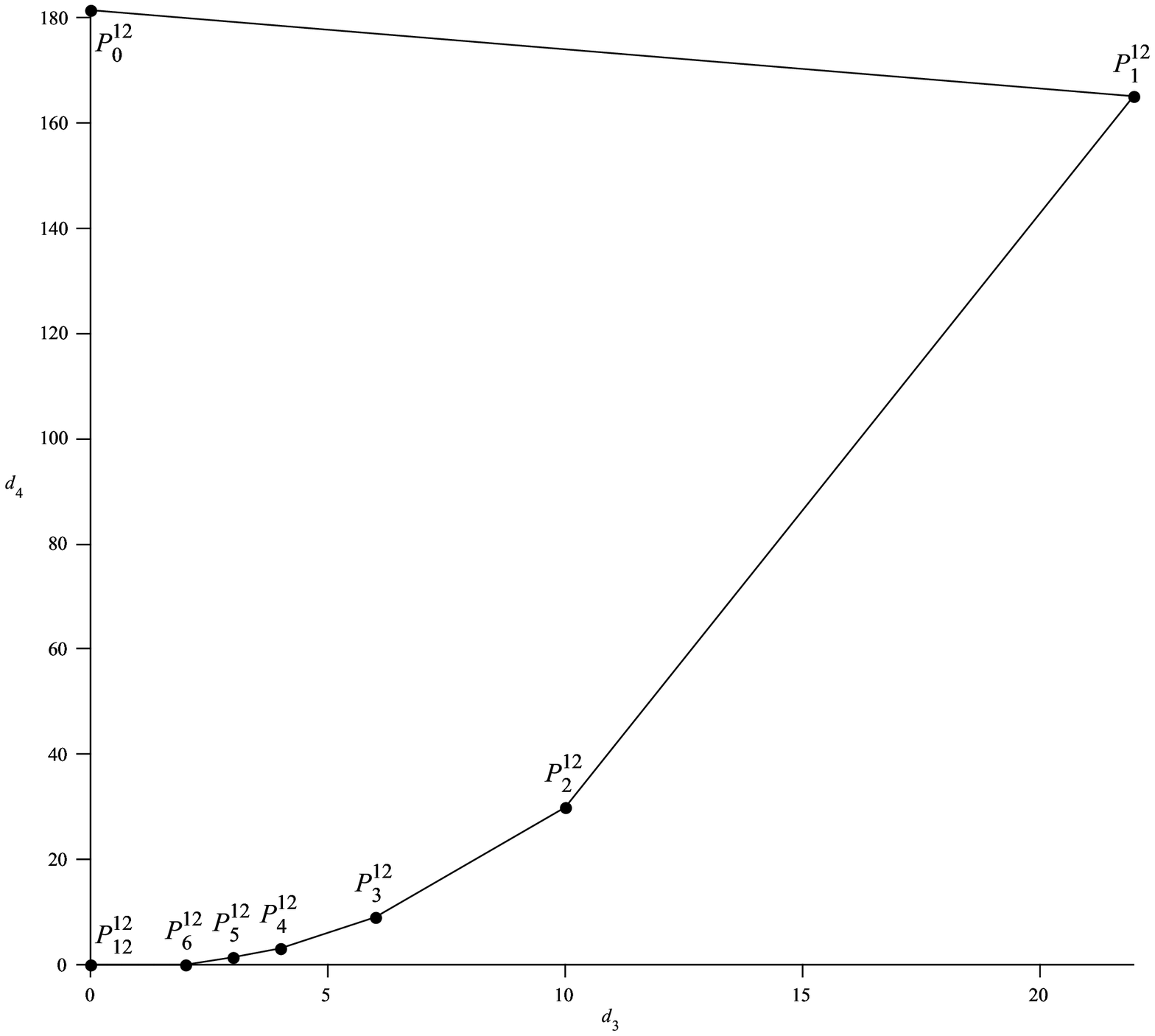}
\caption{$\Q^{12}$}
\label{fig:q12}
\end{figure}

To get a rough picture how $\Q^r$ looks like,
we do the following. For $C^r_l$ we have
$$ d_3(C^r_l) = \frac{1}{6} l \left( \frac{r}{l} \right)^2 +
O(r) \approx \frac{r^2}{6l} \ ; \
d_4(C^r_l) = \frac{1}{8} l \left( \frac{r}{l} \right)^3 +
O(r^2) \approx \frac{r^3}{8l^2} . $$
More precisely, for any fixed $1 \leq l \leq r$
$$ \lim_{r \to \infty} \frac{6}{r^2} \cdot d_3(C^r_l) = \frac{1}{l} \ ; \
\lim_{r \to \infty} \frac{8}{r^3} \cdot d_4(C^r_l) = \frac{1}{l^2} .$$
As for $K_{r,r}$,
$$ d_3(K_{r,r}) = 0 \ ; \ d_4(K_{r,r}) = \frac{r^3}{8} .$$
Consequently, if we consider the image of $\Q^r$ under
the linear transformation that \textit{multiplies}
the $x$-coordinate by $6/r^2$ and the $y$-coordinate by $8/r^3$,
then the obtained polygons converge (in Hausdorff distance) to the following set
$$ \Q \mathdef \conv\left( \left\{(0,0),(0,1) \right\} \cup \left\{
(x,x^2) : \frac{1}{x} \in \IN \right\} \right) \subset \IR^2 .$$
Figure \ref{fig:q_limit} shows the limit set $\Q$.
\begin{figure}[h]
\centering
\includegraphics[width=8.5cm]{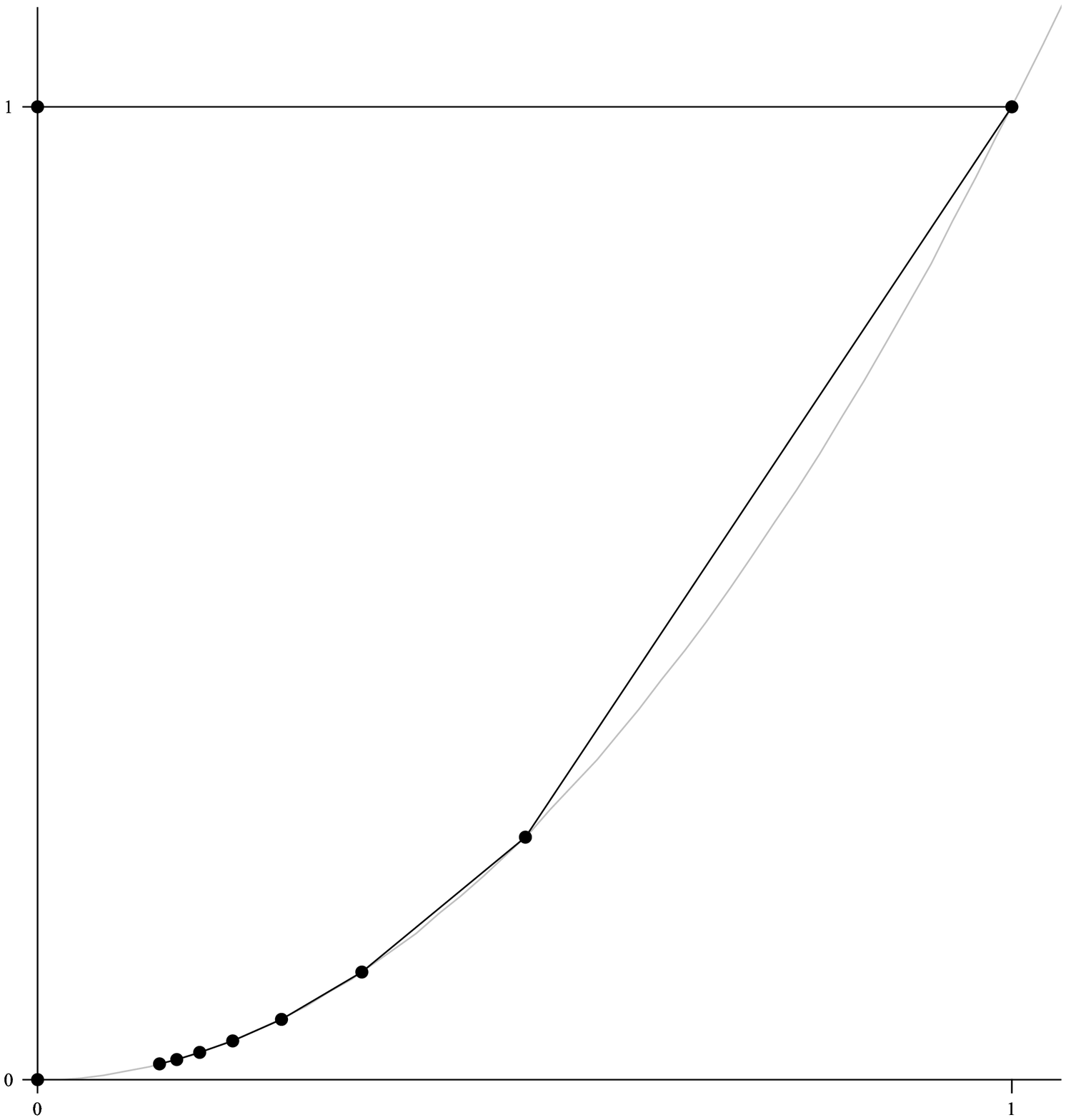}
\caption{$\Q$}
\label{fig:q_limit}
\end{figure}

\subsection{Main theorem}

\begin{theorem} \label{thm:main}
We have $\P^r = \{(x,y) \in \Q^r \, : \, x,y\in \IQ \}$.
In particular, $ cl(\P^r) = \Q^r $.
\end{theorem}
We have already seen that all the vertices of $\Q^r$ lie in $\P^r$.
Due to Proposition \ref{prop:convex} it follows that
$\P^r \supset \{(x,y) \in \Q^r \, : \, x,y\in \IQ \}$.
The points in $\P^r$ have rational coordinates,
thus it suffices to prove that $\P^r \subset \Q^r$.
Taking an arbitrary $r$-regular graph $G$,
we need to show that $P(G) \in \Q^r$.
Since it is clear that $P(G)$ is between the vertical lines $x=0$ and $x=r(r-1)/6$,
we need to prove the following two assertions:
\begin{itemize}
\item $P(G)$ is \textit{under} (or on) the segment
connecting $P^r_0 = P(K_{r,r})$ and $P^r_1 = P(K_{r+1})$.
\item $P(G)$ is \textit{above} (or on) the broken line $P^r_1 P^r_2  \cdots P^r_r$.
\end{itemize}

We know that $P(G)$ is the center of the points $P(G,x)$, $x \in V(G)$.
It turns out that, in fact, all points $P(G,x)$
must lie under the segment $P^r_0 P^r_1$.
(It is quite clear if one thinks about it,
but the rigorous proof is a bit technical, so we skip it now,
but the complete proof can be found in the Appendix,
see Proposition \ref{prop:under}.)
It follows that their center, $P(G)$, also lies under the segment.

We cannot prove the second assertion in the same manner though,
because it is possible to have points $P(G,x)$
strictly under the broken line $P^r_1 P^r_2  \cdots P^r_r$.
But even if we have such points, the rest of the points
will \textit{pull the center back} above the broken line.
We mention that it is not very hard to show that for $r =3,4$ the points
$$ \frac{1}{2}P(G,x) + \frac{1}{2r} \sum_{y\in N(x)} P(G,y) $$
lie above the broken line for an arbitrary vertex $x$.
(The center of these points is also $P(G)$, so this implies
the second assertion for $r=3,4$.)
However, for larger values of $r$ this method does not work.
We need a different approach.

Let us take a $4$-cycle $x_1x_2x_3x_4$ in $G$.
This $4$-cycle was counted once in each $c_4(G,x_i)$, $i=1,2,3,4$.
(Recall that $\sum_{x\in V(G)} c_4(G,x) = 4c(G)$.)
Now we introduce a different way of counting cycles.
How we count a $4$-cycle $x_1x_2x_3x_4$ depends on whether
the \textit{diagonals} $x_1x_3$ and $x_2x_4$ are also edges of $G$.
\begin{itemize}
\item If both $x_1x_3$ and $x_2x_4$ are edges of $G$,
then we count this $4$-cycle once at each $x_i$, $i=1,2,3,4$.
\item If only one of the diagonals, say $x_1x_3$, is an edge in $G$,
then we count this $4$-cycle twice at $x_1$ and $x_3$,
but we will not count it at $x_2$ and $x_4$.
\item If neither $x_1x_3$, nor $x_2x_4$ is an edge in $G$,
then we do not count this $4$-cycle at all.
\end{itemize}
Let $\ct(G,x)$ denote the total number of $4$-cycles counted at a vertex $x$.
Unlike $c_4(G,x)$, $\ct(G,x)$ is determined by $G^x$:
it is the number of cherries in $G^x$ with adjacent leaves
plus twice the number of cherries in $G^x$ with nonadjacent leaves.

We counted each $4$-cycle at most four times, hence
\begin{equation} \label{eq:c&ct}
\sum_{x \in V(G)} \ct(G,x) \leq 4 c_4(G) = \sum_{x \in V(G)} c_4(G,x) .
\end{equation}
Set $\Pt(G,x) \mathdef \left( c_3(G,x)/3, \ct(G,x)/4 \right)$,
and let the center of these points be $\Pt(G)$.
The points $\Pt(G)$ and $P(G)$ lie on the same vertical line,
and $P(G)$ is above $\Pt(G)$ because of \eqref{eq:c&ct}.
It is easy to see that in the case of $G = C^r_l$
the two points coincide: $\Pt(C^r_l) = P(C^r_l) = P^r_l$.
In the remainder of this section we will prove that
for each vertex $x$ the point $\Pt(G,x)$ is above the broken line.
This would imply that $\Pt(G)$ is also above it and so is $P(G)$.

We know that both $c_3(G,x)$ and $\ct(G,x)$ are determined by $G^x$.
Note that the graph $G^x$ has $r$ vertices,
but, apart from that, it can be arbitrary.
So we take an arbitrary graph $H$ with $r$ vertices.
The number of edges is $e=e(H)$, and for $0 \leq i \leq 3$
by $n_i = n_i(H)$ we denote the number of triples of vertices
with the property that there are $i$ edges going between them.
For example, $n_3$ stands for the number of $K_3$'s in the graph.
Then the total number of triples is
\begin{equation} \label{eq:sum}
n_0 + n_1 + n_2 + n_3 = \binom{r}{3} .
\end{equation}
Let us take a triple of vertices and count the number of edges
going between the three vertices, then add up this number for all possible triples.
On the one hand, this sum is clearly $n_1 + 2n_2 + 3n_3$.
On the other hand, we count each edge $r-2$ times (once in each of
the $r-2$ triples that contain both of its endpoints). Consequently,
\begin{equation} \label{eq:weighted_sum}
n_1 + 2n_2 + 3n_3 = (r-2)e .
\end{equation}
Note that the number of cherries with adjacent leaves is $3n_3$
and the number of cherries with non-adjacent leaves is $n_2$.
We obtain that
$$ c_3(G,x) = e(G^x) \ ; \ \ct(G,x) = 3n_3(G^x) + 2n_2(G^x) .$$

For the graphs $G=C^r_l$, the subgraphs $G^x$ are all isomorphic to $D^r_l$.
So we need to show that the point $(e, 3n_3 + 2n_2)$
is always above the broken line connecting the points
$(e(D^r_l), 3n_3(D^r_l) + 2n_2(D^r_l))$, $l=1,\ldots,r$.
Since $n_2(D^r_l) = 0$, we get a stronger result if,
instead of $3n_3 + 2n_2$, we prove the same for
$3n_3 + (3/2)n_2$, or, equivalently, for $2n_3 + n_2$.

Using \eqref{eq:sum} and \eqref{eq:weighted_sum} we have
$$ 2 n_3 + n_2 - n_0 = (n_1 + 2n_2 + 3n_3) -
(n_0 + n_1 + n_2 + n_3) = (r-2)e - \binom{r}{3} .$$
It follows that
$$ 2n_3 + n_2 = n_0 + (r-2)e - \binom{r}{3} .$$
Thus proving it for $2n_3 + n_2$ is the same as proving it for $n_0$.
Regarding the problem in the complement graph we get the following:
prove that the point $(e, n_3)$ (for an arbitrary graph with order $r$)
is above the broken line connecting the points $(e(T^r_l), n_3(T^r_l))$
corresponding to the Tur\'an graphs $T^r_l$, $l=1,\ldots,r$.
This is a result of Bollob\'as from 1976 \cite{Bo_compl_subgr2}.
(The proof, which is a very nice variant of Zykov's symmetrization,
can also be found in \cite[Chapter VI]{Bo_extremal}.)
Actually, he proved this in greater generality: for points
with the first coordinate being the number of $K_p$'s
and the second coordinate being the number of $K_q$'s
in the graph for arbitrary integers $2 \leq p < q \leq r$.
So using this result of Bollob\'as in the special case $p=2$; $q=3$
gives us exactly what we needed;
the proof of Theorem \ref{thm:main} is complete.

As we have seen, our problem turned out
to be related to the following well-studied problem:
what is the minimal number $f(e)$ of
triangles ($K_3$'s) in a graph with $e$ edges?
There are a lot of good bounds for $f(e)$,
and for certain $e$'s even the exact value of $f(e)$ is known.
For details, see \cite{LovSim, min_dens_tri}.


\section{Unimodular random graphs} \label{sec:unimodular}

In this section we explain how the proof of the previous section
can be modified to work for unimodular random graphs.
First we give a brief overview of random graphs,
Benjamini-Schramm convergence, and unimodularity.

We start with introducing the space of rooted graphs.
For a positive integer $D$ let $\RG_D$ denote the set of
rooted graphs $(G,o)$, where $G$ is a connected,
undirected graph $G$ with degree bound $D$
(that is, each vertex has at most $D$ neighbors),
and $o$ is a distinguished vertex, called the root of $G$.
Note that such a rooted graph has finitely or countably many vertices.
Let the \emph{rooted distance} of $(G_1,o_1)$ and $(G_2,o_2)$ be
$1/k$ where $k$ is the maximal integer such that
the $k$-balls around $o_{1}$ and $o_{2}$ are isomorphic
(as rooted graphs). The rooted distance turns $\RG_D$
to a compact, totally disconnected metric space.

A \emph{random graph} is a Borel probability measure on $\RG_D$.
A natural way to get a convergence notion for these random graphs is
to consider the weak topology on the space of Borel probability measures.
(Note that this is a compact space, since $\RG_D$ is a complete metric space.)

Any finite unrooted graph $G$ gives rise to a random graph
by assigning the root of $G$ uniformly randomly
and taking the connected component of the root.
We denote the distribution of this random rooted graph by $\lambda _{G}$.
This observation allows us to define
a convergence notion for finite unrooted graphs as well.
We say that a sequence $G_n$ of finite graphs is
Benjamini-Schramm convergent if the corresponding
random graphs $\lambda_{G_n}$ converge in the weak topology.
The limit of $G_n$ is defined as the weak limit of $\lambda_{G_n}$.

So we can get random graphs as the limit of finite graphs.
The question arises: which random graphs can we get this way?
The only known condition that such random graphs necessarily satisfy
is called unimodularity. (It is open whether this condition is also sufficient;
recall Question \ref{q:big_open}.) We will restrict ourselves to regular random graphs,
since it is somewhat easier to define unimodularity in that special case.
(A random graph is $r$-regular if it is concentrated
on the set of $r$-regular connected rooted graphs.
A random graph is regular if it is $r$-regular for some $r$.)
We take a regular random graph and pick a uniform random neighbor of the root
and consider the directed edge going from the root to this neighbor.
This way we get a probability measure on the space of
connected graphs equipped with a distinguished directed edge.
If we revert this directed edge, we get another probability measure on the same space.
A regular random graph is unimodular if these two measures coincide.
(Note that this is equivalent to unimodularity only in the case of regular random graphs.)

Unimodularity is equivalent to the property called \emph{mass transport principle}.
Let us consider connected graphs (with degree bound $D$)
with an ordered pair of distinguished vertices.
There is a natural topology on the space of such graphs
(similar to the one we defined on $\RG_D$).
A non-negative Borel function on this space
is called a mass transport function:
the function describes how much mass is sent
from the first distinguished vertex to the second one.
If we have a random graph, then it makes sense to talk about
the expected total mass that the root sends out
as well as the expected total mass that the root receives.
The random graph is unimodular if these two values
are equal for arbitrary mass transport function.

For a random graph $\lambda$ let $d_k(\lambda)$ be the expected number of
$k$-cycles containing the root divided by $k$. Set
$$ \Ph^r = \Ph^r_{3,4} \mathdef \left\{ \left( d_3(\lambda), d_4(\lambda) \right) :
\mbox{$\lambda$ is an $r$-regular unimodular random graph} \right\} \subset \IR^2 .$$
Note that for a finite unrooted graph $G$ we have $d_k(G) = d_k(\lambda_G)$.
It follows that $\P^r \subset \Ph^r$. It is also easy to see that
$d_k$ is a continuous function on $\RG_D$.
Consequently, $\Ph^r$ is the continuous image of $\RG_D$.
Since $\RG_D$ is compact, so is $\Ph^r$.
It follows that $\Ph^r \supset \cl (\P^r) = \Q^r$. In fact:
\begin{theorem} \label{thm:uni}
We have $ \Ph^r = \Q^r = \cl(\P^r)$.
\end{theorem}
\begin{remark}
The equality $\Ph^r = \cl(\P^r)$ means that
any $r$-regular unimodular random graph $\lambda$
can be approximated by $r$-regular finite graphs
in the sense that the points $(d_3,d_4)$ corresponding to
the finite graphs can be arbitrarily close to
$\left( d_3(\lambda),d_4(\lambda) \right)$.
\end{remark}
\begin{proof}[Proof of Theorem \ref{thm:uni}]
We need to prove that
$P(\lambda) = \left( d_3(\lambda), d_4(\lambda) \right) \in \Q^r$
for an arbitrary $r$-regular unimodular random graph $\lambda$.

For a connected rooted graph $(G,o)$ let $c_k(G,o)$ denote
the number of $k$-cycles in $G$ containing $o$.
By definition we have
$$ d_3(\lambda) = \int \frac{c_3(G,o)}{3} \, \mathrm{d} \lambda( (G,o) ) \ \mbox{ and }
\ d_4(\lambda) = \int \frac{c_4(G,o)}{4} \, \mathrm{d} \lambda( (G,o) ) .$$
As in the finite setting, for an arbitrary $r$-regular rooted graph $(G,o)$
the point $(c_3/3, c_4/4)$ lies between the vertical lines $x=0$ and $x=r(r-1)/6$
and below the segment $P^r_0 P^r_1$. It follows that the same holds for the point
$P(\lambda) = ( d_3(\lambda), d_4(\lambda) )$.
(Note that here we do not even need unimodularity.)
Finally, to prove that $P(\lambda)$ lies above the broken line $P^r_1 P^r_2 \cdots P^r_r$
we distinguish four different types of $4$-cycles
containing the root in a rooted graph $(G,o)$. The $4$-cycle $oabc$ is
\begin{itemize}
\item of type $1$ if both diagonals $ob$ and $ac$ are edges of $G$;
\item of type $2$ if only $ob$ is an edge of $G$;
\item of type $3$ if only $ac$ is an edge of $G$;
\item of type $4$ if neither $ob$, nor $ac$ is an edge of $G$.
\end{itemize}
The number of $4$-cycles of type $i$ is denoted by $c_{4,i}(G,o)$; $i=1,2,3,4$.
For a unimodular random graph $\lambda$
\begin{equation} \label{eq:mtp23}
\int c_{4,2}(G,o) \, \mathrm{d} \lambda( (G,o) ) =
\int c_{4,3}(G,o) \, \mathrm{d} \lambda( (G,o) ) .
\end{equation}
This equality follows immediately from the mass transport principle:
whenever we have four vertices $x_1, x_2, x_3, x_4$ such that
any two except the pair $(x_2, x_4)$ are connected,
then let both $x_2$ and $x_4$ send a mass $1/2$ both to $x_1$ and to $x_3$.
We also have
$$ c_4(G,o) = c_{4,1}(G,o) + c_{4,2}(G,o) + c_{4,3}(G,o) + c_{4,4}(G,o) .$$
As in the finite case we set
$$ \ct(G,o) =  c_{4,1}(G,o) + 2c_{4,2}(G,o) .$$
Using \eqref{eq:mtp23} it follows that
\begin{equation} \label{eq:c&ct_uni}
\int c_4(G,o) \, \mathrm{d} \lambda( (G,o) ) \geq
\int \ct(G,o) \, \mathrm{d} \lambda( (G,o) ) .
\end{equation}
On the other hand, the point $(c_3/3, \ct/4)$ lies above the broken line
for an arbitrary $r$-regular rooted graph $(G,o)$; the proof goes
exactly the same way as in the finite setting.
Putting this and \eqref{eq:c&ct_uni} together we conclude that
$P(\lambda)$ is above the broken line as desired.
\end{proof}
Using this result we can say something about
the spectral properties of regular unimodular random graphs.
The so-called spectral measure can be associated
to any locally finite, connected rooted graph $(G,o)$;
it is denoted by $\mu_{G,o}$. It can be defined through
the transition operator of the graph \cite{mohar_woess}. 
For our purposes, however, it suffices to know
that it satisfies the following property:
$$ \int x^k \, \mathrm{d} \mu_{G,o}(x) = p_k(G,o) .$$
So for a random graph, the moments of the expected spectral measure are equal to
the expected return probabilities of the simple random walk starting from the root.
For $k=3,4$, these expected return probabilities can be expressed in terms of
the densities $d_3$ and $d_4$ of the random graph. Thus using Theorem \ref{thm:uni}
one can describe the possible pairs of the third and fourth moments of
the expected spectral measure of $r$-regular unimodular random graphs.


\section{Appendix}

While constructing the extreme graphs in Section \ref{sec:finite},
we needed the existence of certain \textit{regular} hypergraphs with large girth.
In \cite{sauer} such hypergraphs were constructed
but not in the generality we need here.
The proof of the next lemma is a straightforward generalization of
a construction due to Erd\H os and Sachs \cite{erdos_sachs, sachs}.
\begin{lemma} \label{lemma:reg_hyp}
For any positive integers $g$, $r$, and any sequence of integers $s_1, \ldots, s_r \geq 2$,
there exists a hypergraph $H$ with the following properties:
\begin{itemize}
\item each vertex $x$ of $H$ is contained by exactly $r$ hyperedges;
\item the sizes of the hyperedges containing $x$ are
$s_1, \ldots, s_r$ for any given vertex $x$;
\item the girth of $H$ is at least $g$;
that is, any Berge cycle of $H$ has length at least $g$.
\end{itemize}
\end{lemma}
\begin{proof}
We prove by \textit{double induction}.
Take any $g_0 \geq 2$ and $r_0 \geq 2$,
and assume that the lemma holds for $g = g_0-1$ and arbitrary $r$,
and also for $g = g_0$, $r=r_0-1$ (and for arbitrary prescribed sizes).

By the inductive hypothesis we have a hypergraph $H_0$ satisfying the conditions of the lemma
for $g = g_0$, $r = r_0-1$ and prescribed sizes $s_1, \ldots, s_{r_0-1}$.
We also have a hypergraph $G$ satisfying the conditions
for $g = g_0-1$, $r = \left| H_0 \right|$ and
each of the $\left| H_0 \right|$ prescribed sizes being $s_{r_0}$.

Having these two hypergraphs, we take $\left| G \right|$ copies of $H_0$,
denoted by $H_1, \ldots, H_{|G|}$. Our hypergraph $H$ will be the disjoint union of
the hypergraphs $H_1, \ldots, H_{|G|}$ with some additional hyperedges.
Suppose that the vertex set of $G$ is $\{1,2,\ldots,\left| G \right|\}$.
For each hyperedge $E$ of $G$ we add a hyperedge to $H$,
which contains one vertex from every $\{H_i : i\in E\}$.
Since $G$ has $\left| H_0\right|$ hyperedges containing a vertex $i$,
we can choose a different vertex from $H_i$ for each of these
$\left| H_0\right| = \left| H_i\right|$ hyperedges.
Then any vertex of $H$ is in exactly one of
the additional hyperedges (which all have size $s_{r_0}$).

We claim that $H$ satisfies all three conditions of the lemma
for $g=g_0$, $r=r_0$ and prescribed sizes $s_1, \ldots, s_{r_0}$.
The first two conditions clearly hold.
To prove the third condition, we take an arbitrary Berge cycle in $H$:
distinct vertices $x_0, x_1, \ldots, x_{k-1}, x_k = x_0$ and
hyperedges $E_0, \ldots, E_{k-1}$ for some $k\geq 3$
such that $\{x_j, x_{j+1} \} \subset E_j$.
If all the vertices in the cycle lie in the same $H_i$,
then the length $k$ of the cycle must be at least $g_0$.
If not, then the cycle starts in one of the $H_i$'s,
it makes a few steps inside $H_i$ using its hyperedges,
then it jumps into another $H_i$ using
one of the additional hyperedges, and so on.
So we can consider the corresponding cycle in the hypergraph $G$
which has girth at least $g_0-1$. Moreover, the cycle in $H$ must take at least one step
inside every new $H_i$ it jumps into. So $k \geq 2(g_0-1) \geq g_0$.
\end{proof}

Finally, as promised, we give a rigorous proof for the following statement.
\begin{proposition} \label{prop:under}
For an arbitrary $r$-regular graph $G$ and an arbitrary vertex $x \in V(G)$,
the point $P(G,x)$ lies under or on the segment $P_0^rP_1^r$.
\end{proposition}
\begin{proof}
Recall that $G^x$ denotes the induced graph on $N(x)$,
where $N(x)$ is the set of neighbors of $x$.
We will denote the vertices in $N(x)$ by $x_1, \ldots, x_r$.
Let the vertex degree of $x_i$ in $G^x$ be $d_i$.
For each $i$ we have $0 \leq d_i \leq r-1$;
we can assume that $d_1 \geq d_2 \geq \ldots \geq d_r$.

Now we consider the subgraph of $G$ which contains those edges
that have one endpoint in $N(x)$ and one endpoint outside $\{x\} \cup N(x)$.
This is a bipartite graph. The degree of $x_i$ is clearly $r-1-d_i$.
By $e_1 \geq e_2 \geq \ldots$ we will denote
the degrees of the vertices outside $\{x\} \cup N(x)$.
The sum of the numbers $r-1-d_i$ clearly equals the sum of $e_j$'s, thus
\begin{equation} \label{eq:sum_of_degrees}
\sum_{i=1}^{r} d_i + \sum_{j} e_j = r(r-1) .
\end{equation}
The above bipartite graph and $G^x$ determine the point $P(G,x)$.
We can forget the rest of the graph.
We can take an arbitrary simple graph
on the vertex set $N(x) = \{x_1, \ldots, x_r\}$
and denote the degree of $x_i$ by $d_i$.
Then we can take an arbitrary simple bipartite graph such that
the first vertex class is $\{x_1, \ldots, x_r\}$
and the degree of $x_i$ is $r-1-d_i$ for each $1\leq i \leq r$.
We can always complement these (by adding new vertices and edges)
into a simple $r$-regular graph.

Instead of $P(G,x) = \left( c_3(G,x)/3, c_4(G,x)/4 \right)$
we will work with the point $\left( c_3(G,x), c_4(G,x) \right)$.
(This is just an affine transformation of our plane.)
As we have seen, $c_3(G,x)$ is the number of edges in $G^x$ and
$c_4(G,x)$ is the number of type 1 cherries plus the number of type 2 cherries:
$$ c_3(G,x) = \frac{1}{2} \left( d_1 + \cdots + d_r \right) \ ; \
c_4(G,x) = \sum_{i=1}^{r} \binom{d_i}{2} + \sum_{j} \binom{e_j}{2} .$$
The left endpoint $P_0^r$ corresponds to the complete bipartite graph $K_{r,r}$,
that is, $d_i = 0$ for each $i$ and $e_1 = \ldots = e_{r-1} = r$.
The right endpoint $P_1^r$ corresponds to
the complete graph $K_{r+1}$, so we have $d_i = r-1$ for each $i$.
Consequently, after the affine transformation
the leftmost point is $\left( 0, (r-1)\binom{r}{2} \right)$, while
the rightmost point is $\left( \binom{r}{2} , r\binom{r-1}{2} \right)$.
We need to prove that $\left( c_3(G,x), c_4(G,x) \right)$
is under or on the segment connecting these two points.
The slope of this segment is $-1$,
so it suffices to prove that the sum of the coordinates
\begin{equation} \label{eq:coord_sum}
c_3(G,x) + c_4(G,x) =
\frac{1}{2} \left( d_1 + \cdots + d_r \right) +
\sum_{i=1}^{r} \binom{d_i}{2} + \sum_{j} \binom{e_j}{2}
\end{equation}
is at most $(r-1)\binom{r}{2}$.

As a first step, we fix the $d_i$'s and maximize $\sum_{j} \binom{e_j}{2}$.
How should we choose our bipartite graph to maximize $\sum_{j} \binom{e_j}{2}$?
The best we can do is the following.
At the beginning, we have the vertices $x_1, \ldots, x_r$ in the first class
and have no vertices in the second class.
We add a vertex to the second class and
connect it to all possible vertices $x_i$,
that is all vertices with $d_i < r-1$.
Then we add another vertex to the second class and connect it
to each vertex $x_i$ with $d_i < r-2$, and so on.
At the end, we get a bipartite graph for which
\begin{equation} \label{eq:e_j}
e_j = \left| \{i: d_i < r - j \} \right| .
\end{equation}
We claim that this is the bipartite graph
for which $\sum_{j} \binom{e_j}{2}$ is maximal.
To prove that we only need to use that
for arbitrary integers $a \geq b \geq 1$ we have
\begin{equation} \label{eq:choose}
\binom{a+1}{2} + \binom{b-1}{2} \geq \binom{a}{2} + \binom{b}{2} + 1.
\end{equation}
Take the bipartite graph for which $\sum_{j} \binom{e_j}{2}$ is maximal.
Assume that $e_1 < |\{i: d_i < r-1 \}|$. Then there is an index $i$
for which $d_i < r-1$ and $x_i$ is not connected to the vertex corresponding to $e_1$.
So we connect them and delete an other edge going out from $x_i$.
By that we increased $e_1$ by one and decreased some other $e_j$ by one.
Since $e_1 \geq e_j$, \eqref{eq:choose} implies
that $\sum_{j} \binom{e_j}{2}$ increased,
which is a contradiction. Thus $e_1 = |\{i: d_i < r-1 \}|$.
Similarly, $e_2$ must equal $| \{i: d_i < r-2 \} |$, and so on.

Now we forget for a moment that we have a graph.
We just take an arbitrary sequence of integers:
$r-1 \geq d_1 \geq d_2 \geq \ldots \geq d_r \geq 0$.
(Of course, not every such sequence corresponds to a graph,
but this does not bother us now.)
We define $e_j$'s as in \eqref{eq:e_j}.
We want to prove that the expression on the right hand side
of \eqref{eq:coord_sum} is at most $(r-1)\binom{r}{2}$.
Take the sequence $\{d_i\}_{i=1}^r$ for which this expression is maximal,
and assume by contradiction that this maximal value
is greater than $(r-1)\binom{r}{2}$.
Let $d_1 = k$ for some integer $1 \leq k \leq r-1$.
By definition we have $e_j = r$ for $1 \leq j \leq r-k-1$,
but $e_{r-k} < r$. We claim that $e_{r-k} < k$.
Let $s$ be the largest positive integer for which
$d_1 = d_2 = \ldots = d_s = k$.
Then we decrease $d_s$ by one so that $d_s = k-1$.
We still have a non-increasing sequence.
The only $e_j$ that changes is $e_{r-k}$: it increases by one.
So if $e_{r-k} \geq k$, then by \eqref{eq:choose} we get
that the value of \eqref{eq:coord_sum} increases by at least $1/2$,
which contradicts our maximality assumption.
It follows that we have $e_j = r$ for $1 \leq j \leq r-k-1$,
but all the other $e_j$'s are less than $k$.
Also, all $d_i$'s are at most $k$.
Due to \eqref{eq:sum_of_degrees} the sum of the other $e_j$'s
and the $d_i$'s is clearly $r(r-1)-r(r-k-1) = rk$.
So we have a few nonnegative integers such that
each of them is at most $k$ and their sum is $rk$.
Then it follows easily from \eqref{eq:choose} that the expression
$$ \sum_{i=1}^{r} \binom{d_i}{2} + \sum_{j \geq r-k} \binom{e_j}{2} $$
is maximal if $r$ of those integers are equal to $k$ and the rest are $0$.
So the above expression is at most $r \binom{k}{2}$,
which yields that the value of \eqref{eq:coord_sum} is at most
$$ \frac{rk}{2} + (r-k-1) \binom{r}{2} + r \binom{k}{2} .$$
By our assumption this is strictly greater than $(r-1)\binom{r}{2}$.
Thus $rk + rk(k-1) > kr(r-1)$. Dividing by $rk$ we get that $k > r-1$, contradiction.
\end{proof}

\bibliographystyle{plain}	
\bibliography{refs}

\end{document}